\documentclass{amsart}
\usepackage{amssymb, latexsym}
\usepackage{graphicx}
\newtheorem{theorem}{Theorem}

\newtheorem{lemma}{Lemma}

\newtheorem{example}{Example}

\begin{document}
\title{ On Products of  Generalized Geometries}
\author{Ralph R. Gomez}
\address{ Department of Mathematics and Statistics\\
		Swarthmore College \\
		Swarthmore, PA 19081}
		
\author{Janet Talvacchia}
\address{ Department of Mathematics and Statistics\\
		Swarthmore College \\
		Swarthmore, PA 19081}

\date{October 11,2014}

\begin{abstract}
In this paper we address what generalized geometric structures are possible on products of spaces that each admit generalized geometries.  In particular we consider, first, the product of two odd dimensional spaces that each admit a generalized almost contact structure, and then subsequently,  the product of an odd dimensional space that admits a generalized almost contact structure and an even dimensional space that admits a generalized almost complex structure. We also draw attention to the relationship of the Courant bracket to the classical notion of normality for almost contact structures.
\end{abstract}

\maketitle

\section{Introduction}\label{S:intro}

The notion of a generalized complex structures, introduced by Hitchin in his paper \cite{[H]} and developed by his student Gualtieri \cite{[G1]},\cite{[G2]} is a framework that unifies both complex and symplectic structures.  These structures exist only on even dimensional manifolds. The odd dimensional analog of this structure, a generalized contact structure, was taken up by  Vaisman \cite {[V1]},\cite{[V2]}  Wade-Poon \cite{[PW]} and Sekiya \cite{[S]}. This framework unifies almost contact, contact, and cosymplectic structures.  In this paper we consider when products of manifolds that themselves admit generalized geometric structures also admit a generalized geometric structure.
 The main results of the paper are the following:
\begin{theorem}\label{T1}
Let $M_1$ and $M_2$ be odd dimensional smooth manifolds each with a generalized almost contact structures $(\Phi_i, E_{+,i}, E_{-,i}) ,\  \  i=1,2$.  Then $M_1\times M_2$ admits a generalized almost complex structure $\mathcal J$.  Further $\mathcal J$ is  a generalized complex structure if and only if both $(\Phi_i, E_{+,i}, E_{-,i})\  \   i=1,2$ are strong generalized contact structures and $[[E_{\pm, i},E_{\mp ,i}]] = 0$.
\end{theorem}

This result generalizes a classical theorem of Morimoto \cite{[M]}.

The next theorem considers products of odd and even dimensional spaces:

\begin{theorem}\label{T2}
Let $M_1$ be an odd dimensional smooth manifold with a generalized almost contact structure $(\Phi, E_+, E_-)$ and let $M_2$ be an even dimensional smooth manifold with a generalized almost complex structure $\mathcal J$.  Then $M_1 \times M_2$ admits a generalized almost contact structure $\Psi$.  Further $\Psi $ is a generalized contact structure if and only if $\Phi$ is a generalized contact structure and $\mathcal J$ is a generalized complex structure. $\Psi $ is  a strong  generalized contact structure if and only if $\Phi$ is a strong generalized contact structure and $\mathcal J$ is a generalized complex structure.
\end{theorem}

In the last section, we look at the relationship of the condition $[[E_{\pm, i},E_{\mp, i}]] = 0$ to the notion of normality for a classical almost contact structure.

\section{Background on Generalized Complex Structures and Generalized Contact Structures}\label{B:Back}
Throughout this paper we let $M$ be a smooth manifold. Consider the big tangent bundle $TM\oplus T^*M$.  We define a neutral metric on $TM\oplus T^*M$ by $  \langle X + \alpha  , Y + \beta \rangle =  \frac{1}{2} (\beta (X) + \alpha (Y) )$ and the Courant bracket by $[[X+\alpha, Y+ \beta ]] = [X,Y] + {\mathcal L}_X\beta -{\mathcal L}_Y\alpha -\frac{1}{2} d(\iota_X\beta - \iota_Y\alpha)$ where $X, Y \in TM$ and $\alpha ,\beta  \in T^*M$. A subbundle of $TM\oplus T^*M$ is said to be involutive if its sections are closed under the Courant bracket.

A generalized  almost complex structure on $M$ is an endomorphism $\mathcal J$ of $TM\oplus T^*M$ such that $\mathcal J + \mathcal J^* =  0 $ and $\mathcal J^2 = - Id$. Since $\mathcal J^2 = -Id$, $\mathcal J$ has eigenvalues $\pm \sqrt{-1}$.  Let $L$ be the $\sqrt{-1}$ eigenbundle of $\mathcal J$.  We say $\mathcal J$ is a generalized complex structure (alternately, we say ${\mathcal J}$ is integrable) if $L$ is involutive.  $L$ is a maximal isotropic with respect to $\langle , \rangle$.  Note that if $ X + \alpha $ is a null vector so is $\mathcal J ( X + \alpha)$.  By adding further null vectors and extending out to a maximal isotropic, we see that maximal isotropics must be even dimensional.  Since $TM$ is a maximal isotropic, we see that $M$ must be even dimensional in order to admit a generalized almost complex structure.

\begin{example}  \cite {[G1]}
Let $(M^{2n}, J)$ be a complex structure.  Then we get an integrable generalized almost complex structure by setting
$$\mathcal J = \left ( \begin{array}{cc}  -J & 0 \\ 0 & J^* \end{array} \right ).$$
\end{example}

\begin{example}  \cite {[G1]}
Let $(M^{2n}, \omega )$ be a symplectic structure.  The we get an integrable  generalized almost complex structure by setting
$$\mathcal J = \left ( \begin{array}{cc}  0 & -\omega^{-1} \\ \omega & 0 \end{array} \right ).$$
\end{example}

The analog of this structure for odd dimensional spaces is a generalized almost contact structure. We use the definition given in \cite{[S]}.  For an odd dimensional manifold $M$, we define a generalized almost contact structure to be a triple $(\Phi, E_+ , E_-)$ where $\Phi $ is an endomorphism of $TM\oplus T^*M$, and $E_+$ and $E_-$ are sections of $TM\oplus T^*M$ which satisfy

\begin{equation}
\Phi + \Phi^{*}=0
\end{equation}

\begin{equation}\label{phi}
\Phi \circ \Phi = -Id + E_+ \otimes E_- + E_- \otimes E_+
\end{equation}

\begin{equation}\label{sections}
 \langle E_\pm, E_\pm \rangle = 0,  \  \    2\langle E_+, E_-  \rangle = 1
\end{equation}

Since $\Phi$ satisfies $\Phi^3 + \Phi =0$, we see that $\Phi$ has $0$ as well as $\pm \sqrt{-1}$ as eigenvalues.  The kernel of $\Phi$ is $L_{E_+} \oplus L_{E_-}$ where $L_{E_\pm}$ is the line bundle spanned by ${E_\pm}$.  Let $E^{(1,0)}$ be the $\sqrt{-1}$ eigenbundle.  Let $E^{(0,1)}$ be the $-\sqrt{-1}$ eigenbundle.

$$
E^{(1,0)} = \lbrace X + \alpha - \sqrt{-1}  \Phi (X + \alpha ) |  \langle E_\pm, X + \alpha \rangle = 0 \rbrace
$$

$$
E^{(0,1)} = \lbrace X + \alpha + \sqrt{-1}  \Phi (X + \alpha ) | \langle E_\pm, X + \alpha \rangle = 0 \rbrace $$

Then we have the complex line bundles
$$L^+ = L_{E_+} \oplus E^{(1,0)}$$
and
$$L^- = L_{E_-} \oplus E^{(1,0)}$$
are maximal isotropics.  We say $(\Phi, E_+, E_-)$ is a generalized contact structure (alternately we say $\Phi$ is integrable) if either of $L^\pm$ is involutive.  We say $(\Phi, E_+, E_-)$ is a strong generalized contact structure (alternately we say $\Phi$ is strongly integrable) if both $L^+$ and $L^-$ are involutive.

\begin{example}  \cite {[PW]}
Let $(\phi ,  \xi, \eta)$ be an almost contact structure on a manifold $M^{2n+1}$.  Then we get a  generalized almost contact structure  by setting
$$ \Phi = \left ( \begin{array}{cc}  \phi & 0 \\ 0 & -\phi^* \end{array} \right ),\  \   E_+ = \xi,\  \   E_-= \eta $$  where $(\phi^*\alpha )(X) = \alpha (\phi (X)), \   X \in TM,\   \alpha \in T^{*}M$.
\end{example}

\begin{example}  \cite {[PW]}
Let $( M^{2n+1}, \eta )$ be a contact manifold.  Let $\xi $ be the corresponding Reeb vector field so that
$$ \iota_\xi d\eta = 0 \  \  \  \eta ( \xi ) = 1.$$
Then $$\rho ( X) := \iota_X d\eta - \eta ( X)\eta$$ is an isomorphism from the tangent bundle to the cotangent bundle.  Define a bivector field by
$$\pi (\alpha , \beta ) := d\eta (\rho^{-1}( \alpha ), \rho^{-1}( \beta ))$$
then we get a generalized almost contact structure by setting
$$ \Phi = \left ( \begin{array}{cc}  0 & \pi \\ d\eta & 0 \end{array} \right ),\  \   E_+ = \eta,\  \   E_-= \xi .$$
It was shown that this example is not strong.

\end{example}

\section{ Products of Manifolds Admitting Generalized Structures}

If $M_1$ and $M_2$ are odd dimensional manifolds admitting generalized almost contact structures $(\Phi_1, E_{+,1}, E_{-,1})$ and $(\Phi_2, E_{+,2}, E_{-,2})$ does the even dimensional manifold $M_1\times M_2$ necessarily admit a generalized almost complex structure?  If so, what can be said about its integrability?   The classical statement of this question was resolved in 1963 by Morimoto \cite{[M]}. He showed that if $M_1$ and $M_2$ are almost contact manifolds each with  an almost contact structure $(\phi_i, \xi_i, \eta_i), \  \  i= 1,2$ then one can define an almost complex structure on the product $M_1\times M_2$ by
$$J(X_1, X_2) = (\phi_1(X_1) - \eta_2(X_2)\xi_1, \phi_2(X_2) + \eta_1(X_1)\xi_2)$$  Moreover, he proved that $J$ is integrable if and only if the almost contact structures $(\phi_i, \xi_i, \eta_i), i= 1,2$  are normal. We prove an analog of this theorem in the generalized setting:
\begingroup
\def\thetheorem{\ref{T1}}
\begin{theorem}
Let $M_1$ and $M_2$ be odd dimensional manifolds  each with a generalized almost contact structure $(\Phi_i, E_{+, i}, E_{-, i}) ,\  \  i=1,2$.  Then $M_1\times M_2$ admits a generalized almost complex structure $\mathcal J$.  Further $\mathcal J$ is integrable  if and only if both $(\Phi_i, E_{+, i}, E_{-, i})\  \   i=1,2$ are strong generalized contact structures and $[[E_{\pm, i},E_{\mp, i}]] = 0$.
\end{theorem}
\addtocounter{theorem}{-1}
\endgroup

Before we prove this theorem, we address some technical lemmas.
\begin{lemma} Let $M$ be an odd dimensional manifold admitting a generalized almost contact structure $(\Phi, E_+, E_-)$.  Then for any $X+\alpha \in TM \oplus T^*M$, $\langle E_\pm, \Phi (X+\alpha) \rangle = 0$.
\end{lemma}
\begin{proof} By Lemma 3.5 in \cite{[S]}, we know that $\Phi(E_\pm)=0$. Using this together with (\ref{phi}) we have for any $X +\alpha \in TM \oplus T^*M$,
\begin{align*}
					0&= \Phi^3(X+\alpha ) +\Phi (X + \alpha ) \\
					  &=\Phi^2(\Phi(X + \alpha )) + \Phi (X+\alpha )\\
					  &=\langle E_+, \Phi (X + \alpha) \rangle E_- + \langle E_-, \Phi (X + \alpha )\rangle E_+.
\end{align*}
Having shown $0 = \langle E_+, \Phi (X + \alpha) \rangle E_- + \langle E_-, \Phi (X + \alpha )\rangle E_+$, we conclude that both terms vanish since $E_+$ and $E_-$ are independent everywhere.

\end{proof}
\begin{lemma} Let $(M, \Phi, E_+, E_-)$ be a generalized almost contact structure.  Then $\Phi $ is strong if and only if  $[[L^+, E^{(1,0)}]] \subset E^{(1,0)}$ and $[[L^-, E^{(1,0)}]] \subset E^{(1,0)}$.
\end{lemma}
\begin{proof}  Assume $\Phi$ is strong. Let $X+\alpha +r_1E_+ \in L^+$ where $X + \alpha \in E^{(1,0)}$ and $r_1$ is any real-valued function.  Let $ Y+\beta  \in E^{(1,0)}$.  Then, since $\Phi$ is strong, $[[ X+\alpha +r_1E_+, Y + \beta]] = W +\rho + r_2 E_+$ for some $W + \rho \in E^{(1,0)}$ and $r_2$ is any real-valued function. Thus $$ \Phi ( [[ X+\alpha +r_1E_+, Y + \beta]]) = \Phi (W+\rho +r_2E_+) = \sqrt{-1} (W+ \rho ) \in E^{(1,0)}.$$  and
$$ \Phi^2 ( [[ X+\alpha +r_1E_+, Y + \beta]]) = \Phi ( \sqrt{-1} (W+ \rho ) ) = -(W + \rho) \in E^{(1,0)}.$$
But we also have that
\begin{align*}
 \Phi^2 ( [[ X+\alpha +r_1E_+, Y + \beta]]) =& -[[ X+\alpha +r_1E_+, Y + \beta]] \cr
&+ E_+ \otimes E_- ([[ X+\alpha +r_1E_+, Y + \beta]] ) \cr &+ E_- \otimes E_+([[ X+\alpha +r_1E_+, Y + \beta]]).
\end{align*}

\noindent Thus
\begin{align*}
 &E_+ \otimes E_- ([[ X+\alpha +r_1E_+, Y + \beta]] ) = 0 \cr &{\rm and}  -[[ X+\alpha +r_1E_+, Y + \beta]] + (E_- \otimes E_+)([[ X+\alpha +r_1E_+, Y + \beta]])  \in E^{(1,0)}.
\end{align*}

Thus $ -(W+\rho +r_2E_+)  + <E_-, W+\rho +r_2E_+>E_+ \in E^{(1,0)}$.

Hence $-(W+\rho) - \frac{1}{2}r_2E_+ \in E^{(1,0)}$.
This implies $r_2 =0$ and $ [[ X+\alpha +r_1E_+, Y + \beta]] \in E^{(1,0)}$ as desired.

A similar argument shows $[[L^-, E^{(1,0)}]] \subset E^{(1,0)}$.

For the other direction, the assumptions $[[L^+, E^{(1,0)}]] \subset E^{(1,0)}$ and $[[L^-, E^{(1,0)}]] \subset E^{(1,0)}$ clearly imply $\Phi $ is strong.

\end{proof}

We list here a criterion for integrability (see \cite{[G1]},\cite{[V1]}) that will be used in the proof of Theorem 1:

Given a generalized almost complex structure $(M, {\mathcal J})$, for $X + \alpha$ and $Y + \beta \in TM \oplus T^*M$, we define the Courant-Nijenhuis tensor
$N_{\mathcal J}(X+ \alpha , Y + \beta ) = [[{\mathcal J}(X+ \alpha ), {\mathcal J} (Y + \beta)]] - {\mathcal J}([[X + \alpha , {\mathcal J}( Y + \beta )]]) - {\mathcal J} ({[[\mathcal J}(X + \alpha ), Y + \beta ]]) + {\mathcal J}^2([[X+ \alpha , Y + \beta ]]).$ We say that ${\mathcal J}$ is integrable if $N_{\mathcal J} = 0$ for all sections of $TM \oplus T^*M$ \cite{[V1]}.\\
Let us now return to Theorem 1:
\begin{proof}[Proof of Theorem 1]We begin by remarking that $T(M_1 \times M_2) \oplus T^*(M_1 \times M_2) \approx (TM_1 \oplus T^*M_1) \oplus (TM_2 \oplus T^*M_2)$. From the generalized almost contact structures on $M_1$ and $M_2$ we can define a generalized almost contact structure $\mathcal J$ on $M_1 \times M_2$ by
\begin{align*}
\mathcal J (X_1 +\alpha_1,X_2 +\alpha_2 ) = &( \Phi_1(X_1+\alpha_1) - 2\langle E_{+,2},X_2 +\alpha_2 \rangle E_{+,1} - 2\langle E_{-,2}, X_2 +\alpha_2 \rangle E_{-,1}, \\
							               &\Phi_2(X_2 +\alpha_2)  + 2\langle E_{+,1}, X_1 +\alpha_1 \rangle E_{+,2} + 2\langle E_{-,1}, X_1 +\alpha_1 \rangle E_{-,2} ).
\end{align*}

Using Lemma 1, and the observation from \cite{[S]} that $\Phi (E_\pm )= 0$ for any generalized almost contact structure $(\Phi, E_+, E_-)$ as well as equations (\ref{phi}) and (\ref{sections}), it is straightforward to compute $\mathcal J^2 = -Id$ and $\mathcal J + \mathcal J^* = 0$.

We see from considering the eigenvalue condition with the explicit formula for $\mathcal J$ given above that the $\sqrt{-1}$ eigenbundle of $\mathcal J$ is generated by

\begin{align*}
&(E_1^{(1,0)}, 0) \\
&(0, E_2^{(1,0)}) \\
&( E_{-,1}, -\sqrt{-1} E_{+,2} ) \\
&( E_{+,1}, -\sqrt{-1} E_{-,2}).
\end{align*}

Assume $(M_1 \times M_2,  {\mathcal J})$ is integrable.  To show $\Phi_{1}$ is strong, we must show that $L^{+}_{1}$
and $L^{-}_{1}$ are closed under the Courant bracket. Since $\sqrt{-1}$ eigenbundle of $\mathcal J$ is closed under the Courant bracket we have that
$$[[(E_1^{(1,0)}, 0), (E_1^{(1,0)}, 0)]] = ([[E_1^{(1,0)},E_1^{(1,0)}]],0) \subset ( E_1^{(1,0)},0) $$
so that $[[E_1^{(1,0)},E_1^{(1,0)}]] \subset E_1^{(1,0)}$.
Also,
$$[[(E_1^{(1,0)}, 0), (E_{+,1}, -\sqrt{-1} E_{-,2} )]] = ([[E_1^{(1,0)},E_{+,1}]], 0) \subset ( E_1^{(1,0)},0)$$
so that $[[E_1^{(1,0)},E_{+,1}]] \subset  E_1^{(1,0)}$. Furthermore, it is clear that $[[L_{E_{\pm,i}},L_{E_{\pm,i}}]]\subseteq L_{E_{\pm,i}}$. Therefore, $L_{1}^{+}$ is closed.
Finally, to argue that $L^{-}_{1}$ is closed it suffices to note that
$$[[(E_1^{(1,0)}, 0), (E_{-,1}, -\sqrt{-1} E_{+,2} )]] = ([[E_1^{(1,0)},E_{-,1}]], 0) \subset ( E_1^{(1,0)},0)$$
obtaining $[[E_1^{(1,0)},E_{-,1}]] \subset  E_1^{(1,0)}$. Hence $\Phi_1$ is strong.

A similar argument using $(0, E_2^{(1,0)})$, $ (E_{+,1}, -\sqrt{-1} E_{-,2} )$, and $( E_{-,1}, -\sqrt{-1} E_{+,2})$ shows that $\Phi_2$ is strong.

To show $[[E_{\pm, i},E_{\mp ,i}]] = 0$, recall that $ {\mathcal J}$ integrable implies $N_{\mathcal J} = 0$.  In particular.
$$( 0,0) = N_{\mathcal J}(( E_{+,1},0), (E_{-,1},0)) = (-[[E_{+,1},E_{-,1}]], [[E_{+,2}, E_{-,2}]]).$$
Thus $[[E_{+,1},E_{-,1}]] = 0$ and $ [[E_{+,2}, E_{-,2}]] = 0$.

To prove the other direction, we assume $\Phi_1$ and $\Phi_2$ are strong and $[[E_{\pm, i},E_{\mp ,i}]] = 0$.  We show $(M_1 \times M_2, {\mathcal J})$ is an integrable generalized complex structure by showing that the generators of the $\sqrt{-1}$-eigenbundle of ${\mathcal J}$ are closed under the Courant bracket. First, observe that we have

$$[[(E_1^{(1,0)},0), (E_1^{(1,0)},0)]] = ([[E_1^{(1,0)}, E_1^{(1,0)}]],0) \subset (E_1^{(1,0)},0)$$ by Lemma 2.

Similarly,

$$[[(E_1^{(1,0)},0), (E_{\pm,1}, -\sqrt{-1}E_{\mp,2})]] = ([[E_1^{(1,0)}, E_{\pm,1}]],0) \subset (E_1^{(1,0)},0)$$


and
$$[[(0,E_2^{(1,0)}), (0,E_2^{(1,0)})]] = (0, [[E_2^{(1,0)}, E_2^{(1,0)}]]) \subset (0,E_2^{(1,0)}).$$

Furthermore,
$$[[(0,E_2^{(1,0)}), (E_{\pm,1}, -\sqrt{-1}E_{\mp,2})]] = (0, [[E_2^{(1,0)},  -\sqrt{-1}E_{\mp,2}]]) \subset (0,E_2^{(1,0)}).$$


Since $[[E_{\pm, i},E_{\mp ,i}]] = 0$, it is straightforward to compute that

$$[[ (E_{+,1}, -\sqrt{-1}E_{-,2}) , (E_{-,1}, -\sqrt{-1}E_{+2})]] = (0,0).$$

Lastly, by direct computation

$$[[(E_1^{(1,0)},0), (0, E_2^{(1,0)})]]\subseteq (E^{(1,0)}_{1},0)\oplus(0,E_{2}^{(1,0)})$$

\end{proof}

If $M_1$ is an odd dimensional manifold with a generalized almost contact structure $(\Phi, E_+, E_-)$ and $M_2$ and even dimensional manifold admitting a generalized complex structure $\mathcal J$ then we can ask if its odd dimensional product $M_1\times M_2$ admits a generalized almost contact structure and, if it does, what one can say about its integrability.  We answer this question in the following:

\begingroup
\def\thetheorem{\ref{T2}}
\begin{theorem}
Let $M_1$ be an odd dimensional manifold with a generalized almost contact structure $(\Phi, E_+, E_-)$.  Let $M_2$ be an even dimensional manifold with a generalized almost complex structure $\mathcal J$.  Then $M_1 \times M_2$ admits a generalized almost contact structure, $\Psi$.  Further $\Psi $ is a generalized contact structure if and only if $\Phi$ is a generalized contact structure and $\mathcal J$ is a generalized complex structure.   $\Psi $ is a strong generalized contact structure if and only if $\Phi$ is a strong  generalized contact structure and $\mathcal J$ is a generalized complex structure.
\end{theorem}
\addtocounter{theorem}{-1}
\endgroup

\begin{proof} Again we note that $T(M_1 \times M_2) \oplus T^*(M_1 \times M_2) \approx (TM_1 \oplus T^*M_1) \oplus (TM_2 \oplus T^*M_2)$.  Define an automorphism $\Psi$ on $(TM_1 \oplus T^*M_1) \oplus (TM_2 \oplus T^*M_2)$ by
$$\Psi (X_1 + \alpha_1, X_2 + \alpha_2 ) = \Phi (X_1 + \alpha_1) + \mathcal J (X_2 + \alpha_2 )$$ for $X_1 + \alpha_1 \in TM_1 \oplus T^*M_1$ and $X_2 + \alpha_2 \in TM_2 \oplus T^*M_2$.
We use the sections $E_+$ and $E_-$ from the generalized almost contact data on $M_1$ to construct the needed sections on $M_1 \times M_2$.

Let $E_+ \in \Gamma   (TM_1 \oplus T^*M_1)$ be given by $X_+ +\alpha_+$ and let  $E_- \in \Gamma   (TM_1 \oplus T^*M_1)$ be given by $X_- +\alpha_-$.

Define $$E_+^{M_1\times M_2} =( \iota_q)_*(X_+ ) + \pi^*\alpha_+$$ and $$E_-^{M_1\times  M_2} =( \iota_q)_*(X_-) + \pi^*\alpha_-$$
 where $\iota_q = (p,q)$ is the standard inclusion map
$\iota: M_1 \rightarrow M_1 \times M_2$ and  $\pi (p,q) = p$ is the standard projection, $ \pi : M_1\times M_2 \rightarrow M_1$.
Then it is straightforward to compute that
 $$\Psi +\Psi^* = 0$$
 $$\Psi^2 = -Id + E_+^{M_1\times M_2} \otimes E_-^{M_1\times M_2} + E_-^{M_1\times M_2} \otimes E_+^{M_1\times M_2} $$
$$\langle E_\pm^{M_1 \times M_2}, E_\pm^{M_1 \times M_2} \rangle = 0$$
$$2 \langle E_+^{M_1 \times M_2}, E_- ^{M_1 \times M_2}\rangle = 1$$
Thus $(\Psi, E_+^{M_1 \times M_2}, E_-^{M_1 \times M_2})$ forms  a generalized almost contact structure on $M_1\times M_2$.

The $\sqrt{-1}$ eigenbundle of $\Psi$ is $E^{(1,0)} \oplus L$ where $E^{(1,0)}$ is the $\sqrt{-1}$ eigenbundle of $\Phi$ and $L$ is the $\sqrt{-1}$ eigenbundle of $\mathcal J$.  The line bundle spanned by  $ E_{+}^{M_1\times M_2}$,  $L_{E_{+}^{M_1\times M_2} }$, is isomorphic to $L_{E_+}$.  Similarly, $L_{E_{-}^{M_1\times M_2} }$ is isomorphic to $L_{E_-}$. So one of
\begin{align*}
L^\pm_{M_1\times M_2} &= L_{E_\pm^{M_1 \times M_2} }+ E^{(1,0) }+ L  \\
&\approx L^\pm_{M_1} + L
\end{align*}
is Courant involutive  if and only if $\Phi$ is a generalized contact structure and $\mathcal J$ is a generalized complex structure. Both $L^+_{M_1\times M_2}$ and $L^-_{M_1\times M_2}$  are Courant involutive if and only if $\Phi$ is a strong generalized contact structure and $\mathcal J$ is a generalized complex structure.
\end{proof}

\section{ Normal Almost Contact Structures and a Geometric Interpretation of $[[E_+ ,E_- ]] = 0$.}

Given a classical almost contact structure $(\phi , \xi, \eta )$ on an odd dimensional manifold $M^{2n+1}$, one can construct an almost complex structure $J$ on $M \times \mathbb R$ (see \cite{[B]})
$$J ( X, f \frac{\partial }{ \partial t})  = (  \phi (X) - f \xi , \eta (X)\frac{\partial }{ \partial t})$$
where $X$ is a vector field on $M$, $t$ is the coordinate on $\mathbb R$ and $f$ is a $C^\infty$ function.  If $J$ is integrable we say the almost contact structure $(\phi, \xi, \eta )$ is normal. Let's consider this situation in the generalized geometry context:

We associate a generalized almost contact structure  $ \Phi = \left ( \begin{array}{cc}  \phi & 0 \\ 0 & -\phi^* \end{array} \right )$, $E_+ = \xi,\  \   E_-= \eta $ to the classical almost contact structure on $M$ as in Example 3.  The real line $ \mathbb{R}$ has a standard classical almost contact structure $(\phi_{\mathbb R} = 0, \frac{\partial }{ \partial t}, dt )$. Associate to it a generalized almost contact structure $ \Phi_{\mathbb R}$,  $ E_{+, {\mathbb R}} = dt$,$E_{-, {\mathbb R}} = \frac{\partial}{\partial t}$.  Using Theorem 1, we get  a generalized almost complex structure on $M\times \mathbb R$  that corresponds to the classical almost complex structure $J$ on $M \times \mathbb R$ given above.  The integrability of $J$ and $\mathcal J$ are determined by whether $\Phi$ is strong and the bracket condition $[[E_+, E_-]] = 0$. These conditions reduce to the classical conditions for normality of the almost contact structure written in terms of $\phi$, $\xi$, and $\eta$. (See \cite {[B]})
\begin{align*}
[\phi , \phi ](X,Y) + 2 d\eta (X, Y) \xi &= 0 \\
(\mathcal L_{\phi X}\eta ) (Y) - \mathcal L_{\phi Y}\eta (X) &=0 \\
\mathcal L_\xi \phi (X) &=0 \\
\mathcal L_\xi \eta (X ) &= 0
\end{align*}
where $X$ and $Y$ are vector fields on $M$. In particular the condition $[[ E_+, E_-]] = 0$ corresponds to the equation $\mathcal L_\xi \eta = 0$.  This  provides an opening for a more geometric understanding of the Courant bracket.

\section*{acknowledgements}
We thank Charles P. Boyer for useful conversations.

The second author thanks the Courant Institute for Mathematical Sciences for their hospitality during the work on this paper.
Both authors would like to thank the referee for helpful suggestions.

\end{document}